\documentclass[11pt]{amsart}
\usepackage{amssymb,latexsym,amsmath,tikz-cd}
\usepackage{enumitem}
\usepackage[mathscr]{eucal}
\usepackage{bbm,xcolor,comment}
\usepackage{mathdots}
\usepackage{chngcntr}

\numberwithin{equation}{section}

\def\Kahler{{K\"ahler}}

\def\cf{cf.~}

\newcommand{\hsp}[1]{{\hbox{\hspace{#1}}}}

\newcommand{\mystack}[2]{\ensuremath{ \substack{ \hbox{\tiny{${#1}$}} \\ \hbox{\tiny{${#2}$}} }} }

\newcounter{letcnt1} 

\newcounter{letcnt2} 
\counterwithin{letcnt2}{letcnt1}

\newcounter{talkcnt} 

\def\a{\alpha}  
\def\b{\beta}  
\def\d{\delta}  
\def\e{\varepsilon}

\def\w{\omega}

\def\GsD{\Gamma\backslash D}
\def\cA{\mathcal{A}}

 \def\tad{\mathrm{ad}}
 
\def\tAut{\mathrm{Aut}}

\def\fb{\mathfrak{b}} 
 
\def\bC{\mathbb C}

\def\td{\mathrm{d}} 
\def\tdeg{\mathrm{deg}} \def\tdet{\mathrm{det}}
 
 \def\tdim{\mathrm{dim}}
\def\cE{\mathcal E}

\def\tEnd{\mathrm{End}}

\def\cF{\mathcal F} 

\def\ff{\mathfrak{f}}

\def\cG{\mathcal G} 

\def\tGr{\mathrm{Gr}}
\def\fg{{\mathfrak{g}}}

\def\cH{\mathcal H}

\def\bi{\mathbf{i}}

\def\fk{\mathfrak{k}}

 \def\sL{\mathscr{L}}

 \def\cN{\mathcal N}

 \def\cO{\mathcal O}
 
\def\sO{\mathscr{O}}

\def\bP{\mathbb P} 
 
 \def\olP{{\overline P}}
\def\fp{\mathfrak{p}}

\def\bQ{\mathbb Q}

\def\bR{\mathbb R} 
 \def\sR{\mathscr{R}}

\def\sS{\mathscr{S}}

\def\bs{\mathbf{s}}

\def\ft{\mathfrak{t}}

\def\fu{\mathfrak{u}}

 \def\cV{\mathcal V}

\def\fv{\mathfrak{v}}

\def\cW{\mathcal W}

   \def\bZ{\mathbb Z}

\def\tand{\quad\hbox{and}\quad}
\def\bs{\backslash}

\def\tinyb{{\hbox{\tiny{$\bullet$}}}}
\def\smallb{{\hbox{\small{$\bullet$}}}}

\def\inj{\hookrightarrow}
\def\sur{\twoheadrightarrow}

\def\op{\oplus}
\def\ot{\otimes}

\def\tw{\hbox{\small $\bigwedge$}}

\newenvironment{sblist}{ 
  \begin{list}{$\smallb$}
   {\usecounter{cnt} \setlength{\itemsep}{2pt}
    \setlength{\leftmargin}{20pt} \setlength{\labelwidth}{20pt}
    \setlength{\listparindent}{20pt} }
   }
   {\end{list}}

\newenvironment{a_list}
  {\begin{enumerate}[label=(\alph*),itemsep=3pt,leftmargin=25pt,listparindent=20pt]}
  {\end{enumerate}}
\newenvironment{a.list}
  {\begin{enumerate}[label=\alph*.,itemsep=3pt,leftmargin=25pt,listparindent=20pt]}
  {\end{enumerate}}

\newenvironment{num.list}
  {
  \begin{enumerate}[itemsep=3pt,leftmargin=25pt,listparindent=20pt,label={\arabic*.}]
  }
  {\end{enumerate}}
\newenvironment{i_list}
  {\begin{enumerate}[label=(\roman*),itemsep=3pt,leftmargin=25pt,listparindent=20pt]}
  {\end{enumerate}}
\newenvironment{i_list_emph}
  {\begin{enumerate}[label=\emph{(\roman*)},itemsep=3pt,leftmargin=25pt,listparindent=20pt]}
  {\end{enumerate}}

\newtheorem{corollary}[equation]{Corollary}
\newtheorem{lemma}[equation]{Lemma}
\newtheorem{proposition}[equation]{Proposition}
\newtheorem{theorem}[equation]{Theorem}

\newtheorem{conjecture}[equation]{Conjecture}
\newtheorem*{theorem*}{Theorem}

\theoremstyle{definition}

\newtheorem*{boldQ*}{Question}
\newtheorem*{boldP*}{Problem}

\theoremstyle{definition}

\theoremstyle{remark}
\newtheorem*{assume*}{Assume}
\newtheorem*{answer*}{Answer}

\newtheorem*{claim*}{Claim}

\newtheorem{definition}[equation]{Definition}
\newtheorem*{definition*}{Definition}
\newtheorem{example}[equation]{Example}
\newtheorem*{example*}{Example}

\newtheorem*{hint*}{Hint}
\newtheorem*{notation*}{Notation}
\newtheorem{remark}[equation]{Remark}
\newtheorem*{remark*}{Remark}
\newtheorem*{remarks*}{Remarks}
\newtheorem*{fact*}{Fact}

\newtheorem*{emphQ*}{Question}
\newtheorem*{emphA*}{Answer}

\def\fibre{A}
\def\fibrezero{{\fibre^0}}
\def\fibreone{{\fibre^1}}

\def\olB{{\overline{B}}}

\def\bfE{\mathbf{E}} 
\def\cEe{\cE_\mathrm{e}}

\def\bfF{\mathbf{F}}
\def\cFe{\cF_\mathrm{e}}
\def\bfH{\mathbf{H}}

\def\bfI{\mathbf{I}}
\def\bfL{\mathbf{L}}

\def\Le{\Lambda_\mathrm{e}}

\def\olO{\overline{\sO}}

\def\olOone{{\overline \sO{}^1}}
\def\olP{{\overline{\wp}}}
\def\olPzero{\olP{}^0}

\def\olPone{\olP{}^1}

\def\tPhi{\widetilde\Phi}
\def\PhiT{\Phi^\mathsf{T}}

\def\Phizero{\Phi^0} 
\def\Phione{\Phi^1}

\def\tsum{{\textstyle{\sum}}}

\def\bfU{\mathbf{U}}
\def\bfV{\mathbf{V}}
\def\cVe{\cV_\mathrm{e}}

\def\Zuo{MR1803724}

\begin{document}
\title[Natural bundles on completions]{Natural line bundles on completions of period mappings}

\author[Green]{Mark Green}
\email{mlg@math.ucla.edu}
\address{UCLA Mathematics Department, Box 951555, Los Angeles, CA 90095-1555}

\author[Griffiths]{Phillip Griffiths}
\email{pg@math.ias.edu}
\address{Institute for Advanced Study, 1 Einstein Drive, Princeton, NJ 08540}
\address{University of Miami, Department of Mathematics, 1365 Memorial Drive, Ungar 515, Coral Gables, FL  33146}

\author[Robles]{Colleen Robles}
\email{robles@math.duke.edu}
\address{Mathematics Department, Duke University, Box 90320, Durham, NC  27708-0320} 
\thanks{Robles is partially supported by NSF DMS 1611939, 1906352.}

\date{\today}

\begin{abstract}
We give conditions under which natural lines bundles associated with completions of period mappings are semi-ample and ample.
\end{abstract}
\maketitle
\setcounter{tocdepth}{1}
\let\oldtocsection=\tocsection
\let\oldtocsubsection=\tocsubsection
\let\oldtocsubsubsection=\tocsubsubsection
\renewcommand{\tocsection}[2]{\hspace{0em}\oldtocsection{#1}{#2}}
\renewcommand{\tocsubsection}[2]{\hspace{3em}\oldtocsubsection{#1}{#2}}

\section{Introduction} \label{S:intro}

We consider pairs $(\olB,Z)$ consisting of a smooth projective variety $\olB$ and a reduced normal crossing divisor $Z \subset \olB$, and suppose that the the complement
\[
  B \ = \ \olB\bs Z
\]
has a variation of (pure) polarized Hodge structure
\begin{equation}\label{iE:vhs}
\begin{tikzcd}[row sep=small,column sep=tiny]
  \cF^p \arrow[r,phantom,"\subset"] 
  & \cV \arrow[r,equal] \arrow[d] 
  & \tilde B \times_{\pi_1(B)} V 
  \\ & B\,. &
\end{tikzcd}
\end{equation}
We assume that the Hodge structures are effective and of weight $n \ge 1$; so that $\cF^p = 0$ for all $p > n$, and $\cF^p = \cV$ for all $p \le 0$.

Let $\cFe^p \subset \cVe$ denote Deligne's extension of the Hodge vector bundles \eqref{iE:vhs} to $\olB$.  Natural line bundles over $\olB$ include the determinants $\tdet(\cFe^p)$, and the  (\emph{extended, augmented}) \emph{Hodge line bundle}
\begin{equation}\label{E:Le}
   \Le \ = \ 
   \tdet(\cFe^n) \ot \tdet(\cFe^{n-1}) \ot \cdots \ot 
  \tdet(\cFe^{\lceil (n+1)/2 \rceil})  \,.
\end{equation}
We let 
\[
   \Lambda \ = \ \left.\Le\right|_B \ = \ 
   \tdet(\cF^n) \ot \tdet(\cF^{n-1}) \ot \cdots \ot 
  \tdet(\cF^{\lceil (n+1)/2 \rceil})
\]
denote the (\emph{augmented}) \emph{Hodge line bundle} over $B$.
\[
   \Lambda \ = \
   \tdet(\cF^n) \ot \tdet(\cF^{n-1}) \ot \cdots \ot 
  \tdet(\cF^{\lceil (n+1)/2 \rceil}) \,.
\]
In addition to the Hodge line bundle $\Le$, natural line bundles for the pair $(\olB,Z)$ include: the normal bundles $[Z_i] = \cN_{Z_i/\olB}$, with $Z_i$ the irreducible components of $Z$, and the log canonical bundle $K_{\olB}+[Z]$. 

Given a line bundle $L \to X$ over a compact complex manifold we consider two groups of properties:
\begin{i_list}
\item Numerical: $L$ is nef; $L$ is big.
\item Geometric: $L$ is semi-ample; $L$ is ample.
\end{i_list}

\begin{example}[Numerical properties]
Quite a bit is known about numerical properties of natural line bundles for the triple $(\olB,Z;\Phi)$:  (a) The Hodge line bundle $\Lambda \to B$ is nef \cite[Proposition 7.15]{MR0282990}.\footnote{See \cite{MR4023377} for nef-ness in characteristic $p$.}  And from the nefness of $\Lambda$ and \cite[Theorem 1.4.1]{GGLR} one may deduce that the extension $\Le\to\olB$ is nef as well. 

(b) Both $\Lambda$ and $\Le$ are big if and only if $\Phi$ satisfies generic local Torelli (\S\ref{S:loctor0}).  

(c) Likewise, assuming generic local Torelli, 
the log canonical bundle $K_\olB + [Z]$ is big \cite{\Zuo}.\footnote{There are also a number of results on the hyperbolicity of $B$, including \cite{BroBru20, Brun20, DLSZ19}.}  
\end{example}

\noindent The geometric properties of semi-ampleness and ampleness are more subtle, and it is these that we are predominately interested in here.  (Though we will identify conditions under which the log canonical bundle is nef.)

\begin{remark}\label{R:BBT}
The semi-ampleness of the augmented Hodge line bundle on the Zariski open subset $B \subset \olB$ was established under additional assumptions by Sommese \cite{MR0324078}, and in general by Bakker--Brunebarbe--Tsimerman \cite{BBT18}.
\end{remark}

\subsection{Period mappings and extensions}

Let
\begin{equation} \label{iE:Phi}
  \Phi : B \ \to \ \GsD
\end{equation}
denote the period map induced by \eqref{iE:vhs}.  Here $D$ is a period domain parameterizing weight $n$, $Q$--polarized Hodge structures on the vector space $V$ (with fixed Hodge numbers), and $\pi_1(B) \sur \Gamma \subset \tAut(V,Q)$ is the monodromy representation.

In order to state our results, we need to recall two extensions 
\begin{equation}\label{E:extns}
\begin{tikzcd}
  \olB \arrow[r,->>,"\Phione"'] \arrow[rr,bend left,"\Phizero"]
  & \olPone \arrow[r,->>] 
  & \olPzero
\end{tikzcd}
\end{equation}
of the period map \eqref{iE:Phi}.  We refer the reader to \cite[\S2]{GGRhatPT} for a detailed discussion of these maps.  Briefly, the normal crossing divisor $Z$ is a finite disjoint union of quasi-projective $Z_I^* \subset Z$ over which the period map induces a variation of nilpotent orbits (via Schmid's nilpotent orbit theorem).  Passing to the weight-graded quotient induces a period map on $Z_I^*$, and $\left.\Phizero\right|_{Z_I^*}$ is this period map (possibly modulo a finite quotient, cf.~\cite[\S\S2.3--2.4]{GGRhatPT}).  The extension $\Phione$ encodes both this period map, and level one extension data in the limiting mixed Hodge structure underlying the nilpotent orbit.  Both $\olPone$ and $\olPzero$ are compact Hausdorff topological spaces containing the image
\[
  \wp \ = \ \Phi(B)
\]
as an open, dense subset.  If we take $Z_W$ to be the disjoint union of the $Z_I^*$ with equivalent weight filtrations, then the restrictions of $\Phizero$ and $\Phione$ to $Z_W$ are proper and analytic.  

\begin{remark} \label{R:gglr}
The image $\olPzero = \Phizero(\olB)$ is conjectured to be algebraic \cite{GGLR}.  In the classical case that $D$ is Hermitian and $\Gamma$ is arithmetic, $\olPzero$ is the closure of $\wp = \Phi(B)$ in the Satake--Baily--Borel compactification $\overline{\GsD}{}^\mathsf{S}$ of $\GsD$.  So in this case, the conjectured algebraicity is immediate.  In general, the quotient $\GsD$ admits no algebraic structure \cite{MR3234111}.  Nonetheless the image $\wp = \Phi(B)$ is quasi-projective (Remark \ref{R:BBT}).  If one could show that $\Le$ is semi-ample on $\olB$ (not just $B$), then it would follow that $\olPzero$ is the projective completion of the quasi-projective $\wp$.\footnote{For applications for moduli, it is not enough to have a projective completion of $\wp$ -- one also wants an extension of the period map.}
\end{remark}

\subsection{The classical story} \label{S:classical}

It is instructive to begin with a review of related results in the ``classical case'' that $D$ is Hermitian and $\Gamma$ is arithmetic.  (An underlying goal of this paper is to develop analogs of these results for period mappings.)  The locally Hermitian symmetric space $\GsD$ is quasi-projective and admits several projective completions.  The Hodge line bundle $\Le$ is defined on these completions and is free.  To be precise, let $\overline{\GsD}^\mathsf{S}$ be the Satake--Baily--Borel compactification.

\begin{theorem}[{\cite{MR0216035, MR0338456}}]\label{T:sbb}
The Hodge line bundle is the pullback of a line bundle $\Lambda \to \GsD$ that extends to an ample line bundle $\Le \to \overline{\GsD}^\mathsf{S}$.  The period map extends to $\Phizero : \olB \to \overline{\GsD}{}^\mathsf{S}$, and the Hodge line bundle $\Le$ is semi-ample over $\olB$.
\end{theorem}

The SBB compactification has toroidal desingularizations
\[
  \overline{\GsD}^\mathsf{T} \ \sur \ \overline{\GsD}^\mathsf{S} \,.
\]
Let 
\[
  \mathrm{bd}^\mathsf{S} \ = \ (\overline{\GsD}^\mathsf{S}) - (\GsD)
  \tand 
  \mathrm{bd}^\mathsf{T} \ = \ (\overline{\GsD}^\mathsf{T}) - (\GsD)
\]
be the corresponding boundaries.  The log canonical bundles associated with these boundaries are semi-ample.

\begin{theorem}[{\cite{MR471627}}] \label{T:mum}
In the case that $D$ is Hermitian and $\Gamma$ arithmetic, 
\begin{eqnarray*}
   \overline{\GsD}^\mathsf{S} & = & 
   \mathrm{Proj} \op_d
   H^0( d\,(K_{\overline{\GsD}^\mathsf{S}} +
    [\mathrm{bd}^\mathsf{S}]))
   \\ & = & 
   \mathrm{Proj} \op_d
   H^0( d\,(K_{\overline{\GsD}^\mathsf{T}} + 
   [\mathrm{bd}^\mathsf{T}])) \,.
\end{eqnarray*}
\end{theorem}

\subsection{Hodge line bundle}

What one would like is to prove that the Hodge line bundle \eqref{E:Le} is semi-ample over $\olB$ (cf.~Remark \ref{R:gglr}).  Outside the classical setting of Theorem \ref{T:sbb}, this is known in only a few special cases (eg.~ \cite{GGLR}).  Nonetheless, we conjecture that an even stronger statement holds when the differential of the period map is generically injective.

\begin{conjecture} \label{iC:LZ}
\emph{(a)}
If the differential of the period map $\Phi : B \to \GsD$ is generically injective, then there are $0\le a_i \in\bQ$ so that the $\bQ$ line bundle $\Le-\sum a_i [Z_i]$ is semi-ample.

\emph{(b)}
Under suitable local Torelli-type assumptions \emph{(Remark \ref{iR:loctor})}, there exist integers $0 \le a_i \in \bZ$ and $m_0$ so that $m \Le - \sum a_i [Z_i]$ is ample for $m \ge m_0$.
\end{conjecture}

\noindent Part (b) of the conjecture is known to hold in two cases: when $Z = Z_1$ is irreducible we have Proposition \ref{P:pg1}; and when $B$ is a surface we have the stronger Theorem \ref{T:dim=2}.  

\begin{proposition}\label{P:pg1}
Suppose that $Z = Z_1$ consists of a single irreducible component, and that $\td\Phione$ is injective on $\Phizero$--fibres.  Assume also that the effective cone $\mathrm{Eff}^1(\olB)$ of 1--cycles is finitely generated.  Then the line bundle $\Pi = m\Le - [Z]$ is ample for $m \ge m_0$.
\end{proposition}

\noindent The proposition is proved in \S\ref{S:prf-pg1}.

For the theorem, index the irreducible components $Z_i$ of $Z = Z_1 \cup \cdots \cup Z_\nu$ so that $\Phizero(Z_i)$ is a point if and only if $i \le \mu \le \nu$.

\begin{theorem}\label{T:dim=2}
Suppose that $\tdim\,B=2$ and assume that the differential of $\Phi : B \to \GsD$ is everywhere injective.  Then there exists $a_i \ge 0$ so that the line bundle 
\[
  \Pi \ = \ m \Le \ - \ \sum_{i=1}^\mu a_i [Z_i]
\]
is ample for $m \gg 0$.
\end{theorem}

\noindent The theorem is proved in \S\ref{S:prfdim=2}.

\begin{remark}[Local Torelli for $(\olB,Z;\Phi)$]\label{iR:loctor}
Conjecture \ref{iC:LZ}(b) alludes to ``suitable local Torelli-type assumptions''.  What we have in mind here is a local Torelli condition for the triple $(\olB,Z;\Phi)$.  Specifically (and with the notations explained in \S\ref{S:dce}) the Gauss--Manin connection $\nabla$ on $\cV$ induces a natural map
\begin{equation}\label{pgi:15}
  \Psi : T_\olB(-\log Z) \ \to \ \tGr_{\cFe}^{-1}\tEnd(\cEe) \,.
\end{equation}
We say that the triple $(\olB,Z;\Phi)$ satisfies the \emph{local Torelli property} if \eqref{pgi:15} is injective.  Local Torelli for the triple $(\olB,Z;\Phi)$ implies local Torelli for the period maps $\Phi$ and $\Phizero|_{Z_W}$.
\end{remark}

\subsection{The log canonical bundle: nef and semi-ample}

Recall that the log canonical bundle $K_\olB + [Z]$ is known to be big under generic local Torelli \cite{MR1803724}.  To this we add

\begin{conjecture}
If the differential of the period map $\Phi : B \to \GsD$ is injective at some point (generic local Torelli holds), then $K_\olB + [Z]$ is nef and big.  If the differential of the period map is injective everywhere (local Torelli holds), then $K_\olB+[Z]$ is semi-ample.
\end{conjecture}

\begin{remark}
By the Base Point Free Theorem \cite{MR1233485}, 
If $K_{\olB+[Z]}$ is both nef and big, then $2(\tdim\,B + 2)!\,(\tdim\,B+2)\,(K_\olB+[Z])$ is semi-ample.
\end{remark}

\begin{theorem}[Theorem \ref{T:pg1}]\label{iT:pg1}
Assume that the local Torelli condition holds for $(\olB,Z;\Phi)$: the bundle map $\Psi : T_\olB(-\log Z) \to \tGr_{\cFe}^{-1}\tEnd(\cEe)$ is injective.  Then the line bundle $K_\olB + [Z]$ is nef and big. 
\end{theorem}

\noindent The local Torelli condition for $(\olB,Z;\Phi)$ enables us to realize the restriction of $K_\olB+[Z]$ to $B$ as a line subbundle of the pull-back $\cH = \Phi^*(\bfH)$ of a homogenous subbundle $\bfH \to D$ (which descends to $\GsD$).  Curvature properties of $\bfH$ imply that $c_1(\cH_\mathrm{e})$ is essentially equivalent to $c_1(\Le)$.  The theorem is then deduced by considering the second fundamental form of $K_\olB+[Z] \inj \cH_\mathrm{e}$.

\subsection{The log canonical bundle: ample}

Recall the extensions \eqref{E:extns}.  Let $\fibrezero$ be a connected component of a $\Phizero$--fibre, and let $\fibreone \subset \fibrezero$ be a connected component of a $\Phione$--fibre.

\begin{theorem}[Theorem \ref{T:pg1ample}]\label{iT:pg1ample}
Assume that the local Torelli condition holds for $(\olB,Z;\Phi)$: the bundle map $\Psi : T_\olB(-\log Z) \to \tGr_{\cFe}^{-1}\tEnd(\cEe)$ is injective.   Suppose in addition that the effective cone $\mathrm{Eff}^1(\olB)$ is finitely generated and that the period maps $\left.\Phizero\right|_{Z_W}$ have constant rank.  Then there is a well-defined Gauss map $\cG(\left.\Phione\right|_\fibrezero) : \fibrezero \to \tGr(r_W,\bC^{d_W})$.   The line bundle $K_\olB + [Z]$ is ample if and only if the Gauss map $\cG(\left.\Phione\right|_\fibrezero)$ is locally injective.
\end{theorem}

\begin{remark}\label{R:asms}
Both the assumption that the effective cone $\mathrm{Eff}^1(\olB)$ is finitely generated (appearing in several statements), and the hypothesis that the period maps $\left.\Phizero\right|_{Z_W}$ have constant rank are expected to be unnecessary.  Our goal here is not to prove the optimal results, but to highlight the key geometric ideas underlying the arguments.
\end{remark}

\tableofcontents 

\subsection*{Acknowledgements}

We are indebted to Ben Bakker and Kang Zuo for enlightening conversations and correspondence.

\section{The case that $\tdim\,B=2$} \label{S:dim=2}

\subsection{Proof of Theorem \ref{T:dim=2}} \label{S:prfdim=2}

Consider the case that $B$ is a surface.  If $\wp = \Phi(B)$ is a curve, then \cite{MR0324078} and \cite{MR1273413} imply that $\olPzero$ is algebraic.  One may then show that $\Le \to \olPzero$ is ample (\cf~the argument of \cite[\S6]{GGLR}). 

So we assume that $\wp$ is a surface; equivalently, $\Phi_*$ is injective at some point $b \in B$.  Index the irreducible components $Z_i$ of $Z = Z_1 \cup \cdots \cup Z_\nu$ so that $\Phizero(Z_i)$ is a point if and only if $i \le \mu \le \nu$.
It follows from \cite[Lemma 5.4.20]{GGLR} that it suffices to prove the following: given a curve $C \subset \olB$, we have 
\begin{equation}\label{E:dim=2}
  \Pi \cdot C \ = \ \tdeg\,\left.\Pi\right|_C 
  \ > \ 0 \,.
\end{equation}
Without loss of generality $C$ is irreducible and there are three cases to consider:
\begin{a_list}
\item \label{i:dim2a}
The intersection $C\cap B$ is Zariski open in $C$.  (In which case $C \cap Z$ is a finite set of points.)
\item \label{i:dim2b}
$C = Z_i$ for some $i>\mu$, and $\Phizero(C)$ is a curve.
\item \label{i:dim2c}
$C = Z_i$ for some $i\le\mu$, and $\Phizero(C)$ is a point.  
\end{a_list}
In cases \ref{i:dim2a} and \ref{i:dim2b}, we have $\Le \cdot C > 0$, cf.~\S\ref{S:loctor0}.  We will see that the $a_i$ are determined by the intersection matrix 
\begin{equation}\label{E:A}
  A \ = \ (A_{ij}) \ = \ \Vert Z_i \cdot Z_j \Vert_{i,j=1}^\mu \,.
\end{equation}
Then \eqref{E:dim=2} will follow for $m\gg0$.

In the case \ref{i:dim2c} we have $\Le \cdot C = 0$.  The Hodge Index Theorem implies that $A$ is negative definite \cite[Lemma 3.1.1]{GGLR}.

\begin{lemma}\label{L:A}
Let $A$ be any integral, negative definite symmetric matrix with the property that $A_{ij} \ge0$ when $i\not=j$.  Then $A$ has an eigenvector $a = {}^t(a_1,\ldots,a_\mu)$ with $a_i > 0$.
\end{lemma}

\noindent The lemma is proved below.  Assuming the lemma for the moment, let $\a < 0$ denote the eigenvalue of $a$.  Then 
\[
  \sum_{j=1}^\mu a_j Z_j \cdot Z_i \ = \ \a\,a_i \ < \ 0 \,,\quad i \le \mu \,.
\]
The desired \eqref{E:dim=2} now follows, and the proof is complete.

\begin{remark}
The point here is that a simpler result, such as $m\Le - \sum_i [Z_i]$ is ample, does not hold.  The coefficients are necessary; they reflect a property of the singularity that $Z_i$ is contracted to.
\end{remark}

\begin{proof}[Proof of Lemma \ref{L:A}]
Suppose that $a = (a_1,\ldots,a_\mu)$ is an eigenvalue with maximal eigenvalue $\a < 0$.  We claim that $\hat a = {}^t( |a_1| , \ldots , |a_\mu| )$ is also an eigenvector with maximal eigenvalue $\a$.  To see this note that 
\begin{eqnarray*}
  \sum_{i,j=1}^\mu |a_i| \, A_{ij} \, |a_j| & = & 
  \sum_i |a_i| \, A_{ii} \, |a_i| 
  \ + \  \sum_{i\not=j} |a_i| \, A_{ij} \, |a_j| \\
  & \ge & \sum_i A_{ii} \, (a_i)^2 
  \ + \  \sum_{i\not=j} a_i \, A_{ij} \, a_j 
  \ = \ \a\,\Vert a \Vert^2 \,.
\end{eqnarray*}
So without loss of generality we may suppose that $a_i \ge 0$.  

We further claim that $a_i > 0$ for all $1 \le i \le \mu$.  Suppose that some $a_j=0$.  Set $a_\epsilon = (a_1 + \epsilon \d_{1j} , \ldots , a_\mu + \epsilon \d_{\mu j})$.  Then 
\begin{eqnarray*}
  {}^ta_\epsilon\,A\,a_\epsilon & = & 
  {}^ta \,A\,a \ + \ 2\epsilon \sum_{i=1}^\mu A_{ij}\,a_i
  \ + \ \epsilon^2\, A_{jj} \,.
\end{eqnarray*}
Since $a_j=0$ we have $\Vert a_\epsilon \Vert^2 = \Vert a \Vert^2 + \epsilon^2$.  This implies
\begin{eqnarray*}
  \frac{{}^ta_\epsilon\,A\,a_\epsilon}{\Vert a_\epsilon \Vert^2} 
  & \ge & 
  \frac{{}^ta \,A\,a}{\Vert a \Vert^2} \ + \ 
  \frac{2\epsilon \sum_{i=1}^\mu A_{ij}\,a_i}{\Vert a \Vert^2}
  \ + \ 
  \frac{\epsilon^2\, A_{jj}}{\Vert a \Vert^2} \,.
\end{eqnarray*}
The connectedness of $Z$ implies that some $A_{ij}\,a_i > 0$.  So for $0 < \epsilon \ll 1$ we have 
\[
  \frac{{}^ta_\epsilon\,A\,a_\epsilon}{\Vert a_\epsilon \Vert^2} 
  \ > \  
  \frac{{}^ta \,A\,a}{\Vert a \Vert^2} \,,
\]
contradicting the maximality of $\a<0$.  Thus, $a_i > 0$ for all $1 \le i \le \mu$.
\end{proof}

\subsection{Remark on negative definiteness of $A$}

In the proof of Theorem \ref{T:dim=2} we invoked the Hodge Index Theorem to conclude that the matrix \eqref{E:A} is negative definite.  Alternatively, one may show (without the Hodge Index Theorem) that \eqref{E:A} is negative definite if the conormal bundle is ample.  To be precise, suppose that $X$ is a surface and that $Z_i \subset X$ are smooth curves forming a normal crossing divisor $Z = Z_1 \cup \cdots \cup Z_\mu$ 

\begin{lemma} \label{L:im}
If $\cN_{Z/X}^* \to Z$ is ample, then the intersection matrix $A = \Vert Z_i \cdot Z_j \Vert_{i,j=1}^\mu$ is negative definite.
\end{lemma}

\begin{proof}
Let $\a$ be a maximal eigenvector; we wish to show that $\a < 0$.  The argument of Lemma \ref{L:A} applies here to give us an eigenvector $(a_1,\ldots,a_\mu)$ with $a_i > 0$.

The condition that $\cN_{Z/X}^* \to Z$ be ample is 
\[
  \sum_{i=1}^\mu Z_i \cdot Z_j \ < \ 0
  \quad \forall \ 1 \le j \le \mu \,.
\]
We have 
\[
  \sum_{i=1}^\mu a_i\, Z_i \cdot Z_j \ = \ \a\,a_j
  \quad \forall \ 1 \le j \le \mu \,.
\]
Without loss of generality the $Z_i$ are indexed so that $a_1 \ge a_i$.  Then
\[
  \sum_{i=1}^\mu a_i Z_i \ = \ 
  a_1 \sum_{i=1}^\mu Z_i \ - \ \sum_{i=2}^\mu (a_1-a_i) Z_i \,,
\]
so that 
\[
  \a\,a_1 \ = \ \sum_{i=1}^\mu a_i\, Z_i \cdot Z_1
  \ = \ 
  a_1 \sum_{i=1}^\mu Z_i \cdot Z_1 
  \ - \ \sum_{i=2}^\mu (a_1-a_i) Z_i \cdot Z_1 \,.
\]
Since $Z_i \cdot Z_1 \ge 0$ for all $i \ge 2$, our ampleness hypothesis implies that $\a\,a_1 < 0$.
\end{proof}

\begin{remark}
The converse to Lemma \ref{L:im} does not hold.  For example, consider
\[
  A \ = \ \left[ \begin{array}{rr}
      -5 & 2 \\ 2 & -1
  \end{array} \right] \,.
\]
\end{remark}

\subsection{Constraints on variations over $Z$}

We finish our discussion of the $\tdim\,B=2$ case with the discussion of a constraint on the variations of nilpotent orbits that may arise; the result is an application of the central geometric information \cite[(4.5)]{GGRhatPT} that arises when studying the variation along $\fibrezero$.

\begin{proposition} \label{P:dim=2}
Assume that $\tdim\,B=2$ and that $\Phi : B \to \GsD$ satisfies generic local Torelli (equivalently, $\Phi_*$ is injective at some point $b \in B$, so that $\tdim\,\wp=2$).  Then $\Phione$ is necessarily non-constant on some irreducible component $Z_i$ of $Z$.
\end{proposition}

\begin{definition} \label{dfn:ht}
The variation $(\cW^I,\left.\cFe\right|_{Z_I^*})$ of limiting mixed Hodge structure along $Z_I^*$ is of \emph{Hodge--Tate type}, if the associated graded variation $\cFe^p(\tGr^{\cW^I}_a)$ of Hodge structure is Hodge--Tate.
\end{definition}

\begin{remark} \label{R:ht}
When the variation is of Hodge--Tate type, both the period map $\Phizero$ and the level one extension data map $\Phione$ are locally constant along $Z_I^*$.  
\end{remark}

\begin{corollary}\label{C:dim=2}
Suppose that $B$ is a surface and that the LMHS along all of $Z$ is of Hodge--Tate type.  Then $\tdim\,\wp \le 1$.
\end{corollary}

\begin{proof}[Proof of Proposition \ref{P:dim=2}]
We argue by contradiction.  Suppose that $\Phione$ is constant along all of $Z$.  Then $\Phizero$ is necessarily constant along all of $Z$; that is, $Z = \fibrezero$, and $\mu = \nu$.  Since $\Phione(Z) = \Phione(\fibrezero)$ is a point in the compact torus $T_W$ of \cite[Theorem 4.3]{GGRhatPT}, it follows from \cite[(4.5)]{GGRhatPT} that 
\[
  (\left.\Phione\right|_Z)^*(\sL_M)
  \ = \ \sum_{i=1}^\nu \left. \kappa(M,N_i)[Z_i]\right|_Z
  \quad\hbox{is trivial.}
\]
So 
\begin{eqnarray*}
  0 & = & \left( \sum_{i=1}^\nu \kappa(M,N_i)[Z_i] \right)^2 
  \ = \ \sum_{i,j=1}^\nu \kappa(M,N_i) \kappa(M,N_j) \, Z_i \cdot Z_j
\end{eqnarray*}
The negative definiteness of \eqref{E:A} forces $\kappa(M,N_i) = 0$ for all $i$.  As $M$ is arbitrary, this contradicts \cite[(4.18)]{GGRhatPT}.
\end{proof}

\section{Geometric properties of $K_{\olB} + [Z]$} \label{S:K+Z}

The purpose of this section is to establish conditions under which $K_{\olB}+[Z]$ is semi-ample and ample.  The principle assumption is a local Torelli condition on the triple $(\olB,Z;\Phi)$ (Definition \ref{dfn:ltc-trip}).

\begin{theorem}\label{T:pg1}
Assume that the local Torelli condition holds for $(\olB,Z;\Phi)$: the bundle map $\Psi : T_\olB(-\log Z) \to \tGr_{\cFe}^{-1}\tEnd(\cEe)$ is injective.  Then the line bundle $K_\olB + [Z]$ is nef and big. 
\end{theorem}

\begin{proof}[Outline of proof]
\begin{num.list}
\item
The Hodge line bundle $\Le$ is nef \eqref{E:HLBnef}.  And it is big when the period map $\Phi$ satisfies generic local Torelli, \cite[\S5.3.2]{GGRhatPT}).
\item
Local Torelli for $(\olB,Z;\Phi)$ implies local Torelli for $\Phi$ (Lemma \ref{L:ltc2}).
\item 
Lemma \ref{L:VIII.12a} which asserts that there exists a positive constant $\epsilon$ so that $c_1(K_\olB + [Z]) \ge \epsilon c_1(\Le)$.
\end{num.list}

\noindent
The local Torelli condition for $(\olB,Z;\Phi)$ enables us to realize the restriction of $K_\olB+[Z]$ to $B$ as a line subbundle of the pull-back $\cH = \Phi^*(\bfH)$ of a homogenous subbundle $\bfH \to D$ (which descends to $\GsD$).  Curvature properties of $\bfH$ imply that $c_1(\cH_\mathrm{e})$ is essentially equivalent to $c_1(\Le)$ (Lemma \ref{L:VIII.10}).  Lemma \ref{L:VIII.12a} is then deduced by considering the second fundamental form of $K_\olB+[Z] \inj \cH_\mathrm{e}$.
\end{proof}


\begin{theorem}\label{T:pg1ample}
Assume that the local Torelli condition holds for $(\olB,Z;\Phi)$: the bundle map $\Psi : T_\olB(-\log Z) \to \tGr_{\cFe}^{-1}\tEnd(\cEe)$ is injective.   Suppose in addition that the effective cone $\mathrm{Eff}^1(\olB)$ is finitely generated and that the restriction $\Phizero_W$ of $\Phizero$ to $Z_W$ has constant rank.  Then there is a well-defined Gauss map $\cG(\left.\Phione\right|_\fibrezero) : \fibrezero \to \tGr(r_W,\bC^{d_W})$.   The line bundle $K_\olB + [Z]$ is ample if and only if the Gauss map $\cG(\left.\Phione\right|_\fibrezero)$ is locally injective.
\end{theorem}

\noindent The Gauss map is defined in \eqref{E:Gauss}. 

\begin{remark}\label{R:eff}
The assumption that $\mathrm{Eff}^1(\olB)$ is finitely generated is expected to be unnecessary (Remark \ref{R:asms}).  Here it is a technical convenience: we will show that for each curve $C \subset \olB$, there exists $m_0(C)$ so that $(m\Le-[Z]) \cdot C > 0$ for all $m \ge m_0(C)$.  Finite generation of $\mathrm{Eff}^1(\olB)$ allows us to assume that $m_0(C)$ is independent of $C$.  This will suffice to establish ampleness.
\end{remark}
 
\begin{proof}[Outline of proof]
The proof of Theorem \ref{T:pg1ample} takes off from the end of that for Theorem \ref{T:pg1}.  What remains is to show that $c_1(K_\olB+[Z])$ is positive if and only if the differential of the Gauss map is injective; this is \eqref{E:tau} and Lemma \ref{L:VIII.12b}.  See \S\ref{S:prfTpg1}.
\end{proof}


There is an interesting subtlety in Theorem \ref{T:pg1}, as illustrated by the following\footnote{We are indebted to Kang Zuo for bringing Example \ref{eg:Zuo} to our attention.}

\begin{example}\label{eg:Zuo}
Let $\overline\cA_g{}^\mathsf{T}$ be a toroidal compactification of the moduli space $\cA_g$ of principally polarized abelian varieties.  Then $K_{\overline\cA_g{}^\mathsf{T}} + [\mathrm{bd}^\mathsf{T}]$ is not ample (Theorem \ref{T:mum}).  The point here is that, following \cite{MR471627}, we assume that $\overline\cA_g{}^\mathsf{T}$ is smooth, and that $Z = \mathrm{bd}^\mathsf{T}$ is a local normal crossing divisor.\footnote{The actual singularities of $\overline\cA_g{}^\mathsf{T}$ are mild quotient singularities that do not affect the argument.  Similarly, the fact that $Z$ is only a local, rather than a global, normal crossing divisor makes no essential difference.}  If $\overline\cA_g{}^\mathsf{S}$ is the Satake--Baily--Borel compactification of $\cA_g$, then the natural map
\begin{equation}\label{E:zuo}
  \pi : \overline\cA_g{}^\mathsf{T} \ \to \ \overline\cA_g{}^\mathsf{S}
\end{equation}
is a resolution of the singularities of $\overline\cA_g{}^\mathsf{S}$.  There is an ample line bundle $\cO_{\overline\cA_g{}^\mathsf{S}}(1) \to \overline\cA_g{}^\mathsf{S}$ satisfying $K_{\overline\cA_g{}^\mathsf{T}} + Z = \pi^*(\cO_{\overline\cA_g{}^\mathsf{S}}(1))$ \cite[(3.4)]{MR471627}.  The fibres of \eqref{E:zuo} can be identified with the abelian varieties $J_I$ of \cite[Theorem 4.3]{GGRhatPT}; consequently, the Gauss map \eqref{E:Gauss} of $\Phione$ on these fibres is constant.
\end{example}

\subsection{Local Torelli for the period map and the Chern form $c_1(\Le)$} \label{S:loctor0}

The Chern form $c_1(\Lambda) \in \cA^{1,1}_\olB$ of the Hodge line bundle $\Lambda \to B$ has the property that 
\begin{equation}\label{E:HLBnef}
  c_1(\Lambda)(v,\bar v) \ = \ \|\Phi_{\ast}(v)\|^2 \,,
\end{equation}
for all $v \in TB$, \cite[Proposition 7.15]{MR0282990}.  The product $\w^{\tdim\,B}$ is non-negative, and positive at those points $b \in B$ where $\td \Phi_b$ is injective.   At infinity $c_1(\Lambda)$ extends to a $(1,1)$--current $c_1(\Le)$ on $\olB$ where it represents the Chern class of the extended $\Le \to \olB$ \cite{MR840721}.  In particular, the line bundle $\Le$ is big if and only if the period map $\Phi : B \to \GsD$ satisfies generic local Torelli.

The restriction of $c_1(\Le)$ to $Z_I^*$ is well-defined, and as an element of $\cA^{1,1}_{Z_I^*}$ and represents the Chern class of the Hodge line bundle $\Lambda_I = \left.\Le\right|_{Z_I^*}$ of the for the period map $\left.\Phizero\right|_{Z_I^*}$, \cite[Theorem 1.4.1]{GGLR}; more generally, one may show that $\w$ is well-defined on $Z_W$, \cite{GGpos}.  If $v \in T_bZ_I^*$, then \eqref{E:HLBnef} implies
\begin{equation}\label{E:c1on0}
  c_1(\Le)(v,\bar v) \ = \ \|\Phizero_{I,\ast}(v)\|^2 \,.
\end{equation}
In particular, $\Le$ is nef.  (See also \cite[Lemma 6.4]{BBT18}.)  And the differential of the period map $\Phi_I : Z_I^* \to \Gamma_I\bs D$ is every where injective if and only if $\left.c_1(\Le)\right|_{Z_I^*}$ is positive.  Thus, if $C \subset \olB$ is an irreducible curve and $\Phizero(C)$ is not a point, then 
\[
  \Le \cdot C \ = \ \tdeg \left. \Le \right|_C \ = \ \int_C c_1(\Le) \ > 0 \,.
\]

\subsection{Proof of Proposition \ref{P:pg1}} \label{S:prf-pg1}

It suffices to show that there exists $m_0$ so that $(\Le-[Z])\cdot C > 0$ for all curves $C \subset \olB$ and $m \ge m_0$.  Without loss of generality, we may assume that $C$ is an irreducible curve.  

If the image $\Phizero(C)$ is also a curve, then \S\ref{S:loctor0} implies that $\Le \cdot C > 0$.  So we will have $(m\Le - [Z]) \cdot C > 0$ when $m \gg0$.  Now suppose that $C \subset \fibrezero$ is contained in a $\Phizero$--fibre.  Again, \S\ref{S:loctor0} implies that $\Le \cdot C=0$.  However, the hypothesis that $\td\Phione_W$ is injective implies that $\left.\cN^*_{Z / \olB}\right|_C$ is ample, cf.~the discussion following \cite[Theorem 4.3]{GGRhatPT}.  In particular, $-[Z]\cdot C >0$.  The proposition now follows from Remark \ref{R:eff}.
\hfill\qed

\subsection{A local Torelli condition for $(\olB,Z;\Phi)$} \label{S:dce}

One of the technical hypothesis needed for our applications is a local Torelli condition for $(\olB,Z;\Phi)$.  The condition is expressed in terms of the extension of the Gauss--Manin connection (Lemma \ref{L:ltc2}).

Recall Deligne's extension $\cFe^p \to \olB$ of the Hodge bundles \cite{MR1416353}.  Let 
\[
  \cEe^p \ = \ \cFe^p/\cFe^{p+1} \ = \ \tGr^p_{\cFe} \,,
\]
and consider the associated graded vector space 
\[
  \cEe \ = \ \op\,\cEe^p \,.
\]
The Gauss--Manin connection induces a bundle map
\begin{equation}\label{E:Psi}
  \Psi : T_\olB(-\log Z) \ \to \ \tGr^{-1}_{\cFe}(\tEnd(\cEe))\,.
\end{equation}
We review the definition of $\Psi$ in \S\S\ref{S:Psi}--\ref{S:horiz}.

\begin{remark} \label{R:ltc}
At points $b \in B$ the map $\Psi$ is the differential of the period map
\[
  \td\Phi_b \ = \ \left.\Psi\right|_{T_{B,b}} \,.
\]
\end{remark}

\begin{definition}\label{dfn:ltc-trip}
We say that \emph{local Torelli condition holds for $(\olB,Z;\Phi)$} when \eqref{E:Psi} is injective.
\end{definition}


Lemma \ref{L:ltc2} is a generalization of Remark \ref{R:ltc}.

\begin{lemma}\label{L:ltc2}
The local Torelli condition holds for $(\olB,Z;\Phi)$ if and only if 
\begin{i_list_emph}
\item The differential $\td \Phione_I : T (Z_I^*) \to T(\Gamma_I\bs D^1_I)$ is injective for all $I$.
\item
The $\{N_i \ | \ i \in I\}$ are linearly independent for all $I$.
\end{i_list_emph}
\end{lemma}

\noindent The lemma is proved in \S\ref{S:prfltc2}.

\subsubsection{Maurer--Cartan form} \label{S:Psi}

The composition of the lift \cite[(5.14)]{GGRhatPT} with the map \cite[(B.9)]{GGRhatPT} defines
\[
  X \circ \tPhi_\fibreone : \tilde\sO{}^1 \ \to \ \ff^\perp \,.
\]
It will be convenient to set 
\[
  \xi = \exp(X\circ\tPhi_\fibreone) : \sS \ \to \ \exp(\ff^\perp)  \,.
\]
Keep in mind that, since $\ff^\perp$ is a nilpotent subalgebra, the exponential defines a biholomorphism $\ff^\perp \simeq \exp(\ff^\perp)$ so that $X \circ \tPhi_\fibreone$ and $\xi$ carry the equivalent information.

Fix a MHS $(W,F)$ arising along $\fibreone$.  Fix a basis $\{v_j\}$ of $V$ so that every $v_j$ is contained in some $V^{p_j,q_j}_{W,F}$; equivalently, $v_j \in F^{p_j}$ but $v_j \not\in F^{p_j+1}$, and $v_j \in W_{p_j+q_j}$ but $v_j \not\in W_{p_j+q_j-1}$.  Then 
\[
  \phi_j \ = \ \xi \cdot v_j 
\]
defines a framing of $\cVe$ that is adapted to the Hodge filtration $\cFe^p \subset \cVe$.  The key point here is that \cite[Proposition 5.1]{GGRhatPT} implies that the Deligne's construction \cite{MR1416353} applies to this slightly more general setting.  So we may identify the $\{ \phi_j \ | \ p_j = p \}$ with a holomorphic framing of $\cEe^p$.

The pullback 
\[
  \theta \ = \ \xi^{-1} \td \xi
\]
under $\xi$ of the Maurer--Cartan form on $\exp(\ff^\perp)\subset G_\bC$ is $\ff^\perp$--valued.  Let $\xi^i_j$ be the matrix coefficients of $\xi$ with respect to the basis $\{v_j\}$; that is, $\xi \cdot v_j = \xi_j^i v_i$.  Likewise, let $\theta^i_j = (\xi^{-1})^i_k \td \xi^k_j$ be the matrix entries of
\[
  \theta = \theta^i_j\,v_i \ot v^j \,.  
\]
The Gauss--Manin connection satisfies
\[
  \nabla \phi_j \ = \ \theta_j^i \ot \phi_i \,,
\]
and
\begin{equation} \label{E:Psi1}
  \left.\Psi\right|_{\olOone} \ = \ \theta^i_j\,\phi_i \ot \phi^j \,.
\end{equation}

It is instructive to review the proof of the well-known

\begin{lemma} \label{L:logmc}
The pull-back of the Maurer-Cartan form is a log 1-form; that is, 
\[
  \theta^i_j \in \Omega^1_{\olOone}(\log Z\cap\olOone) \,.
\]
\end{lemma}
 
\noindent This is done in \S\ref{S:prflogmc}.  First we review the IPR in this context.

\subsubsection{Horizontality} \label{S:horiz}

Since $\theta$ takes value in $\ff^\perp = \fg_{W,F}^{-,\tinyb}$, we have 
\[
  \theta^i_j \ = \ 0 \quad\forall\quad p_i-p_j > 0 \,.
\]
Horizontality asserts that 
\begin{subequations} \label{SE:thetaX}
\begin{equation}\label{E:theta1}
  \theta^i_j \ = \ 0 \quad\forall\quad p_i-p_j \not= -1 \,.
\end{equation}
And this implies
\begin{equation}
  \theta^i_j \ = \ \td (X \circ \tPhi_\fibreone)^i_j 
  \quad\forall\quad p_i-p_j = -1\,,
\end{equation}
with $X^i_j$ the matrix entries of $X$ with respect to the basis $\{v_j\}$ (defined by $X v_j = X^i_j v_i$); in the notation of \cite[\S B.4]{GGRhatPT}, this is equivalently the statement that
\begin{equation}
  \theta \ = \ \td (X \circ \tPhi_\fibreone)^{-1,\tinyb} \,.
\end{equation}
\end{subequations}
Letting $\theta^{-1,q}$ denote the component of $\theta$ taking value in $\fg^{-1,q}_{W,F}$, we have 
\[
  \theta \ = \ \theta^i_j\,v_i \ot v^j
  \ = \ \tsum\,\theta^{-1,q} \,, 
\]
and
\[
  \theta^{-1,q} \ = \ 
  \sum_{\mystack{p_i-p_j=-1}{q_i-q_j=q}}
  \theta^i_j\,v_i \ot v^j \,.
\]
As a consequence we obtain \eqref{E:Psi}.

\subsubsection{Proof of Lemma \ref{L:logmc}} \label{S:prflogmc}

It follows from \eqref{SE:thetaX} that it suffices to prove 
\begin{equation} \label{E:dX}
  \td (X \circ \tPhi_\fibreone)^{-1,\tinyb} \ \in \ \Omega^1_{\olO{}^1}(\log Z \cap \olO{}^1) \,.
\end{equation}
The 1-form $\td (X \circ \tPhi_\fibreone)^{-1,\tinyb}$ is the differential of the horizontal component of the period matrix.  In particular, the $\td\e_\mu$ of \cite[\S5.3.1]{GGRhatPT} are the matrix entries of $\td (X \circ \tPhi_\fibreone)^{-1,\tinyb}$.  So \eqref{E:dX} is equivalent to \cite[(5.17)]{GGRhatPT} \hfill\qed

\subsubsection{Proof of Lemma \ref{L:ltc2}}\label{S:prfltc2}

The coordinates $\td\e_\mu$ of the map $\Psi_1$ defined in \cite[(5.18)]{GGRhatPT} are the horizontal coordinates of $\left.\Psi\right|_{\olOone}$.  In particular, \cite[Lemma 5.19]{GGRhatPT} is equivalent to Lemma \ref{L:ltc2}.
\hfill\qed

\subsection{Geometric properties of $K_\olB + [Z]$}\label{S:K+Z-geom}

\subsubsection{Geometric implications of local Torelli for $(\olB,Z;\Phi)$} \label{S:ltcgi}

Suppose that the local Torelli condition holds for $(\olB,Z;\Phi)$ (Definition \ref{dfn:ltc-trip}).  Then we may identify $T_{\olB}(-\log Z)$ with a subbundle of $\tGr_{\cFe}^{-1}\tEnd(\cEe)$, and
\[
  (K_\olB + [Z])^* \ = \ \tw^{\tdim\,B} T_{\olB}(-\log Z) 
\]
with subbundle of 
\[
  \cH \ := \ \tw^{\tdim\,B} \tGr_{\cFe}^{-1}\tEnd(\cEe) \,.
\]
The singular metrics on the Hodge bundles $\cEe^p \to \olB$ induce singular metrics on $\tGr_{\cFe}^{-1}\tEnd(\cEe)$ and $\cH$.  The singularities, curvatures and Chern forms of these metrics are much studied \cite{MR840721, MR946244, GGpos, GGLR}.  Over $B$ the metrics and Chern forms are smooth; the Chern forms extend to currents on $\olB$ where they represent the extended vector bundles.  (As already discussed in \S\ref{S:loctor0}, the analogous statements hold for the Hodge line bundle.)  

Let $\Theta_\cH$ denote the curvature matrix of $\cH$, and 
\[
  c_1(\cH) \ = \ \tfrac{\bi}{2\pi} \mathrm{tr}\,\Theta_\cH
\]
the first Chern form.  Let 
\[
  c_1(\Le) \ = \ \tfrac{\bi}{2\pi} \Theta_{\Le}
\]
denote the Chern form of the Hodge line bundle $\Le\to\olB$ (\S\ref{S:loctor0}).

The line bundle $(K_\olB + [Z])^*$ inherits a singular metric from its containment in $\cH$.  Let 
\[
    c_1(K_{\olB}+[Z]) \ = \ \tfrac{\bi}{2\pi} \Theta_{K_{\olB}+[Z]}
\]
denote the Chern form of $K_{\olB}+[Z]$.  We will see that $c_1((K_\olB + [Z])^*)$ is related to $c_1(\cH)$ by the second fundamental form of $(K_\olB + [Z])^* \inj \cH$ (Lemma \ref{L:VIII.12a}), and this will give us control over the $c_1(K_{\olB}+[Z])$.

\begin{lemma}\label{L:VIII.10}
The curvature form $\Theta_\cH$ is nonpositive and there exist positive constants $\epsilon,\epsilon'$ so that 
\[
  \epsilon \,c_1(\Le) \ \le \ -c_1(\cH) \ \le \ \epsilon'\,c_1(\Le) \,.
\]
\end{lemma}

\noindent The lemma is proved in \S\ref{S:prf10}.

\begin{lemma}\label{L:VIII.12a}
There exists a non-negative $\tau \in \cA^{1,1}_B$ that extends to a current on $\olB$, and a positive constant $\epsilon$ so that 
\[
  \epsilon\,c_1(\Le) \,+\, \tau \ \le \ c_1( K_\olB + [Z] ) \,.
\]
\end{lemma}

\begin{proof}[Sketch of Proof]\label{sp:VIII.12a}
The lemma is proved in \S\ref{S:prf12a}.  Here we outline the underlying geometric ideas.  Over $B$ the curvature forms $\Theta_\cH$ and $\Theta_{(K_{\olB}+[Z])^*}$ are related by the second fundamental form 
\begin{equation}\label{E:dfnUp}
  \Upsilon : \cA^0_B((K_{\olB}+[Z])^* ) \ \to \ 
  \cA^{1,0}_B(\cH/(K_{\olB}+[Z])^*) \,,
\end{equation}
which measures the failure of the Chern connection on $\cH$ to preserve the subbundle $(K_{\olB}+[Z])^*$.  We have 
\begin{equation}\label{E:Up}
  \left.\Theta_\cH\right|_{(K_\olB + [Z])^*} \ = \ 
  \Theta_{(K_\olB + [Z])^*} \ + \ (\Upsilon,\Upsilon) \,,
\end{equation}
with $(\Upsilon,\Upsilon)$ a $(1,1)$--form constructed from $\Upsilon$ and the Hermitian metric.  Setting
\[
  \tau \ = \ \tfrac{\bi}{2\pi} (\Upsilon,\Upsilon)\,,
\]
we have 
\begin{equation}\label{E:tau}
  \tfrac{\bi}{2\pi} \left.\Theta_\cH\right|_{(K_\olB + [Z])^*}
  \ = \ -c_1(K_\olB+[Z]) \,+\, \tau \,.
\end{equation}
It then remains to show that there exists a positive constant $\epsilon$ so that 
\[
  \epsilon\,c_1(\Le) \ \le \ 
  -\tfrac{\bi}{2\pi}\left.\Theta_\cH\right|_{(K_\olB + [Z])^*}\,.
\]
We will see that this is a consequence of homogeneity, the structure of the curvature of vector bundles on $D$ and the IPR.
\end{proof}

Note that \eqref{E:c1on0}, Lemma \ref{L:VIII.10} and \eqref{E:tau} imply
\begin{equation}\label{E:tauon0}
  \left.c_1(K_{\olB} + [Z])\right|_\fibrezero 
  \ = \ \left.\tau\right|_\fibrezero \,.
\end{equation}
So to establish the ampleness of $K_{\olB} + [Z]$ we will need to show that 
$\left.\tau\right|_\fibrezero$ is positive.  For this we make the simplifying assumption (expected to be unnecessary, Remark \ref{R:asms}) that the proper, analytic map 
\begin{equation}\label{E:PhiW}
  \Phizero_W: Z_W \ \to \ \wp^0_W \ \subset \ \Gamma_W\bs D_W^0
\end{equation}
of \cite[\S\S2.3--2.4]{GGRhatPT} has constant rank.  Then $\Phizero_W : Z_W \to \wp_W^0$ is a fibration with fibre $\fibrezero$.  Recall that the restriction $\left.\Phione\right|_\fibrezero$ takes value in a compact torus $T_W$, \cite[Theorem 4.3]{GGRhatPT}.  The local Torelli condition for $(\olB,Z;\Phi)$ implies that the differential of $\left.\Phione\right|_\fibrezero$ is injective (Lemma \ref{L:ltc2}).  So we have a Gauss map
\[
  \fibrezero \ \to \ \tGr(r_W , T(T_W)) \,.
\]  
Since $T_W$ is a torus, we may translate each tangent space $T_x(T_W)$ to a fixed $T_e(T_W) = \bC^{d_W}$.  In this way we obtain a Gauss map
\begin{equation}\label{E:Gauss}
  \cG(\left.\Phione\right|_\fibrezero) : \fibrezero \ \to \ 
  \tGr(r_W , \bC^{d_W})
\end{equation} 
to a fixed Grassmannian.

\begin{lemma}\label{L:VIII.12b}
Suppose that the map \eqref{E:PhiW} has constant rank.  Then the restriction $\left.\Upsilon\right|_\fibrezero$ may be identified with the differential of the Gauss map \eqref{E:Gauss}.  In particular, the restriction $\left.c_1(K_\olB+[Z])\right|_\fibrezero$ of the Chern form to $\fibrezero$ is well defined, and it is positive if and only if the differential of the Gauss map $\cG(\left.\Phione\right|_\fibrezero)$ is injective (equivalently, the Gauss map is finite to one).
\end{lemma}

\begin{proof}[Sketch of Proof]\label{sp:VIII.12b}
The lemma is proved in \S\ref{S:prf12b}.  Given \eqref{E:tauon0}, the essential content of the argument is that the restriction of $\Upsilon$ to the the fibre $\fibrezero$ may be identified with the differential of the Gauss map \eqref{E:Gauss}.
\end{proof}

\subsubsection{Proof of Theorem \ref{T:pg1ample}} \label{S:prfTpg1}

Local Torelli for $(\olB,Z;\Phi)$ implies generic local Torelli for $\Phi$ (Lemma \ref{L:ltc2}).  It follows from \S\ref{S:loctor0} and Lemma \ref{L:VIII.12a} that $K_{\olB} + [Z]$ is nef and big.

It remains to establish ampleness.  Following Remark \ref{R:eff}, we need to show that
\begin{equation}\label{E:ample}
  C \cdot (K_\olB + [Z]) \ > \ 0 \,,
\end{equation}
for every irreducible curve $C \subset \olB$.  There are two cases to consider.  \smallskip

\emph{Case 1:  The image $\Phizero(C)$ is a curve}.   Lemma \ref{L:VIII.12a} and \S\ref{S:loctor0} imply
\[
  C \cdot (K_\olB + [Z]) \ = \ \int_{C \cap B} c_1(K_\olB + [Z])
  \ \ge \ \epsilon \int_C c_1(\Le) \ > \ 0 \,,
\]
yielding the desired \eqref{E:ample}.

\smallskip

\emph{Case 2:   The image $\Phizero(C)$ is a point}.  This is the most interesting case.  We necessarily have $C \subset \fibrezero \subset Z$.  Then Lemma \ref{L:VIII.12b} yields
\[
  C \cdot (K_\olB + [Z]) \ = \ \int_C c_1(K_\olB + [Z])
  \ > \ 0 \,,
\]
establishing the desired \eqref{E:ample}.  \hfill\qed

\appendix

\section{Curvature in Hodge theory} \label{S:curvature}

\subsection{Tangent bundle}\label{S:TD}

Recall the Killing form $\kappa$ of $\fg$, \cite[\S B.1.2]{GGRhatPT}.  Define a Hermitian inner product $h_\varphi$ on $T_\varphi D \simeq\op_{p>0} \,\fg^{-p,p}_\varphi$ by $h_\varphi(x,y) = -\kappa( \varphi(\bi) x , \overline{y})$.  This Hermitian inner product is $K^0$--invariant, and so determines a $G_\bR$--invariant Hermitian metric $h$ on $TD$; that is, we have a homogeneous, Hermitian, holomorphic vector bundle
\[
  (TD , h) \ = \ G_\bR \times_{K^0} (T_\varphi D , h_\varphi ) \,.
\]  
The curvature 2-form $\Theta_D \in \cA^{1,1}(D,\fk_\bC^0)$ of this metric is of the form 
\begin{equation}\label{E:ThetaD1}
  \Theta_D \ = \ -\sum_{0>p\,\mathrm{odd}}A_p \wedge {}^t\overline{A_p}
  \ + \ \sum_{0>p\,\mathrm{even}} A_p \wedge {}^t\overline{A_p}
\end{equation}
for some matrices $A_p$ of holomorphic 1-forms with the property that $A_p(\varphi)$ vanishes on every $\fg^{q,-q}_{\varphi}$ with $p\not=q$ \cite[Theorem 4.13]{MR0259958}.  We say that $\Theta_D$ is the difference of disjoint positive $(1,1)$--forms.  We briefly review the construction of $A_p$.  Since these forms are homogeneous, it suffices to determine $A_p$ at the point $\varphi$.  The key observations that are applied below are that (i) the Hodge decomposition $\fg_\bC = \oplus\, \fg^{p,-p}_\varphi$ of \cite[(B.1)]{GGRhatPT} is polarized by $-\kappa$, and (ii) the identity
\[
  \kappa( [x,y] \,,\, z) \ = \ \kappa( x \,,\, [y,z]) \,,
\]
for all $x,y,z\in \fg_\bC$.
\begin{sblist} 
\item
We may choose a compact Cartan subalgebra $\ft_\bR \subset \fk^0_\bR$ of $\fg_\bR$.  Let $\sR \subset \ft^*_\bC$ denote the roots of $\fg_\bC$.  Given a root $\a \in \sR$, let $\fg_\a \subset \fg_\bC$, be the corresponding root space.  
If $x_\a \in \fg_\a$, then $\overline{x_\a} \in \fg_{-\a}$.  So we define $\bar\a = -\a$.  Fix root vectors $x_\a$ so that $\overline{x_\a} = x_{-\a}$.  We may scale the $x_\a$ so that $-\kappa(x_\a,x_{-\b}) = (-1)^{p_\a}\,\d_{\a\b}$, where $p_\a \in \bZ$ is defined by $\fg_\a \subset \fg^{p_\a,-p_\a}$.  
\item
Let $\vartheta \in \Omega^1(G_\bC,\fg_\bC)$ be the component of the left-invariant Maurer-Cartan form taking value in $\op_{p<0}\, \fg^{p,-p}_\varphi \simeq T_\varphi D$.  At the point $\varphi \in D$,
\[
  \Theta_D \ = \ -[\vartheta , \overline\vartheta]_{\fk^0_\bC}
  \ \in \ \tw^{1,1} (T_\varphi D \op \overline{T_\varphi D})^* 
  \,\ot\, \fk^0_\bC\,.
\]
\item
We may define $\vartheta^\a \in \Omega^1(G_\bC)$ by
\[
  \sum_{p_\a<0} \vartheta^\a\,x_\a \ = \ \vartheta \,.
\]
Employing the identification 
\[
  T_\varphi D \ = \ T_\varphi \check D \ = \ 
  \fg_\bC/\fp_\varphi \ \simeq \ \op_{p>0} \,\fg^{-p,p}_\varphi \,,
\]
the forms $\{ \vartheta^\a \ | \ p_\a < 0 \}$ are a basis of the holomorphic cotangent space $T^*_\varphi D$.  Likewise, $\{ \vartheta^{-\a} = \overline{\vartheta^\a} \ | \ \a \in \sR_p \,,\ p < 0 \}$ is a basis of the conjugate $\overline{T_\varphi D}{}^*$.
\item
Let $\sR_p \subset \sR$ denote those roots $\a$ with $p = p_\a$.  Write $\vartheta = \sum_{p<0} \vartheta_p$ with $\vartheta_p = \sum_{p_\a=p} \vartheta^\a\,x_\a$ the component of the Maurer--Cartan form taking value in $\fg^{p,-p}$.   Then
\begin{equation}\label{E:ThetaD2}
  \Theta_D \ = \ -\sum_{p<0}[\vartheta_p , \overline{\vartheta_p}] \,.
\end{equation}
\item
We have chosen root vectors $x_\a$ so that $-\kappa(x_\a , \overline{x_\b}) = (-1)^{p_\a} \d_{\a\b}$.  Fix a basis $\{ x_1,\ldots,x_r\}$ of $\ft_\bR$ so that $-\kappa(x_i,x_j) = -\kappa(x_i,\overline{x_j}) = \d_{ij}$, and set $p_i = 0$ and $\overline{i} =i$.  Then $\{ x_\mu \} = \{ x_\a \}_{\a\in\sR} \cup \{ x_i\}_{i=1}^r$ is a basis of $\fg_\bC$ satisfying $-\kappa(x_\mu , \overline{x_\nu}) = (-1)^{p_\mu}\d_{\mu\nu}$.  Then the matrix representation $M_\a = (M_{\a\mu}^\nu)$ of $\tad(x_\a):\fg_\bC \to \fg_\bC$ with respect to this basis is defined by $\tad(x_\a)x_\mu  = M_{\a \mu}^\nu x_\nu$, and satisfies ${}^t\overline{M_\a} = (-1)^{p_\a+1}M_{\overline\a}$.  (The conjugate transpose ${}^t\overline{M_\a}$ is defined with respect to $h_\varphi$.)  So if we identify $\vartheta_p$ with the matrix of one forms $A_p = \sum_{p_\a=p}\,\vartheta^\a \,M_\a$, then \eqref{E:ThetaD1} holds.
\end{sblist}

\subsection{Hodge bundles} \label{S:not-vb2}

Many of the vector bundles considered over $B$ are the pullbacks (under the period map $\Phi$) of homogeneous holomorphic vector bundles defined on $\check D \supset D$.  

\subsubsection{}
For example,  $\check D$ parameterizes filtrations $F^\tinyb$ of $V_\bC$, the trivial bundle $\check D \times V_\bC$ admits a canonical filtration
\[
  \bfF^n \ \subset \ldots \subset \ \bfF^1 \ \subset \bfF^0
\]
by homogeneous holomorphic vector bundles
\begin{equation}\label{E:bfF}
 \begin{tikzcd}[column sep = tiny,row sep=small]
  \bfF^p \arrow[d] \arrow[r,equal] & G_\bC \times_{P} F^p \\
  \check D \,.
\end{tikzcd}
\end{equation}
and $\cF^p = \Phi^*(\bfF^p)$.  (Here, $P \subset G_\bC$ is the stabilizer of a flag $F^\tinyb \in \check D$.  We may assume with out loss of generality that $F = \varphi \in D$.)  Likewise, the Hodge line bundle $\Lambda = \Phi^*(\bfL)$ with 
\[
  \bfL \ = \ \tdet(\bfF^n) \ot \tdet(\bfF^{n-1}) \ot \cdots \ot 
  \tdet(\bfF^{\lceil (n+1)/2 \rceil}) \,.
\]
Define
\[ \begin{tikzcd}[column sep = tiny,row sep=small]
  \bfE^p \arrow[r,equal] \arrow[d] &
  \bfF^p/\bfF^{p+1} \arrow[r,equal] & G_\bC \times_{P} (F^p/F^{p+1}) \\
  \check D \,.
\end{tikzcd} \]
Both $\left.\bfL\right|_D$ and $\left.\bfE^p\right|_D$ admits Hermitian metrics, via the identification (as smooth vector bundles)
\[ \begin{tikzcd}[column sep = tiny,row sep=small]
  \bfE^p \arrow[r,equal,"\sim"] &
  \bfV^{p,q} \arrow[d] \arrow[r,equal] & G_\bR \times_{K^0} V^{p,q}_\varphi \\
  & D \,.
\end{tikzcd} \]
By definition $h_\varphi(u,v) = Q(\varphi(\bi) u , \overline v)$ defines a Hermitian inner-product on the Hodge summand $V^{p,q}_\varphi$.  This inner-product is invariant under the action of $K^0 \subset \tAut(V^{p,q})$, and so determines a homogeneous Hermitian vector bundle
\[
  (\bfV^{p,q} , h ) \ = \ G_\bR \times_{K^0} ( V^{p,q}_\varphi , h_\varphi ) \,.
\]

\subsubsection{}

The curvature $\Theta_{\bfE^p}$ is also given by an expression similar to \eqref{E:ThetaD1}, \cite[(5.3)]{MR0282990}.  In fact, the curvature forms $\Theta_D$ and $\Theta_{\bfE^p}$ are even more closely related than this might suggest.  

In general, the bundles $\bfU \to D$ considered in Hodge theory are all of the following type: the are homogeneous, Hermitian vector bundles
\[ \begin{tikzcd}[column sep = tiny,row sep=small]
  (\bfU,h) \arrow[d] \arrow[r,equal] & G_\bR \times_{K^0} (U,h_\varphi) \\
  D\,,
\end{tikzcd} \]
with $U$ a $K^0$ submodule of some Hodge representation $G \to \tAut(\tilde U,Q)$, and a Hermitian inner product $h_\varphi$ induced by a polarization (on $\tilde U$ and then restricted to $U$).  They are also the restriction to $D$ of homogeneous, holomorphic vector bundles 
\[ \begin{tikzcd}[column sep = tiny,row sep=small]
  \bfU \arrow[d] \arrow[r,equal] & G_\bC \times_{P} U \\
  \check D\,.
\end{tikzcd} \]
In each of these cases the resulting curvature form is 
\begin{equation}\label{E:ThetaU}
  \Theta_\bfU \ = \ -[\vartheta,\overline{\vartheta}]_{\fu}
  \ = \ -\rho_U ( [\vartheta,\overline{\vartheta}]_{\fk_\bC^0} ) 
  \ = \ \rho_U(\Theta_D) \,,
\end{equation}
where $\fu$ is the image of the Lie algebra representation $\rho_U : \fk^0_\bC \to \tEnd(U)$.  That is, one may think of the various $\Theta_U$ as different matrix representations of the same underlying, endomorphism valued 2-form $\Theta_D$.  (Some care must be taken when $\rho_U$ is not faithful.)  

\subsection{Chern forms}

As \eqref{E:ThetaU} suggests, there is a certain sense in which the associated first Chern forms
\[
  c_1(\bfU,h) \ = \ 
  \tfrac{\bi}{2\pi}\,\mathrm{tr}\,\Theta_\bfU \ \in \ \cA^{1,1}_D
\]
are all the same; at least when the representation $\rho_U$ is faithful.  This is made precise in the following lemma.  Essentially they agree up to a constant when restricted to an irreducible invariant subbundle of $TD$.  The key point is to observe that the $c_1(\bfU,h)$ are all $G_\bR$--invariant.  So their value at an arbitrary $\tilde\varphi \in D$ is determined by their value at a fixed $\varphi \in D$.  So the issue is to show that they agree up to a constant when restricted to an irreducible $K^0$ submodule $\fv \subset T^\bR_\varphi D$ of the (real) tangent space at $\varphi$.  That restriction $\left.c_1(\bfU,h)\right|_\fv$ is $K^0$ invariant.  The same is true of the nondegenerate 
\[
  c_1(TD) \ = \ \tfrac{\bi}{2\pi} \mathrm{tr}\, \Theta_D  \,.
\]

\begin{lemma} \label{L:cherna}
Let $\fv \subset T^\bR_\varphi D$ be any irreducible $K^0$--submodule.  Then there exists a constant $\epsilon(U,\fv) \in \bR$ so that the restrictions satisfy
\[
  \left. c_1(\bfU,h) \right|_\fv \ = \ \epsilon(U,\fv)\,
  \left. c_1(TD) \right|_\fv \,.
\]
\end{lemma}

\begin{proof}
This is a consequence of Schur's lemma and the nondegeneracy of $c_1(TD)$, \cite[Ch.~13]{MR3727160}.
\end{proof}

\begin{remark}
We note that when the weight $n=1,2$, then the horizontal subspace 
\[
  \bfI_\varphi \ = \fg^{-1,1}_\varphi \ \subset \ T_\varphi D
\]
is an irreducible $K^0$--module.  So in this case Lemma \ref{L:cherna} asserts that any $c_1(\bfU,h)$ arising naturally in Hodge theory is a multiple of the \Kahler~form when restricted to the horizontal subspace.
\end{remark}

\subsection{Curvature forms under the IPR} \label{S:not-vb3}

Over the compact dual, the horizontal (homogeneous holomorphic) sub-bundle is
\[ \begin{tikzcd}[column sep = tiny,row sep=small]
  \bfI \arrow[d] \arrow[r,equal] 
  & G_\bC \times_P (F^{-1}\fg_\bC/F^0\fg_\bC) \ \subset \ T\check D\\
  \check D .
\end{tikzcd} \]
When restricting to $D$ we may identify this with
\[ \begin{tikzcd}[column sep = tiny,row sep=small]
  \bfI \arrow[d] \arrow[r,equal] 
  & G_\bR \times_{K_0} (\fg_\varphi^{-1,1}) \ \subset \ T D\\
  D ,
\end{tikzcd} \]
and we have 
\begin{equation}\label{E:ThetaI}
  \left.\Theta_D\right|_{\bfI} \ = \ -A_1\wedge{}^t\overline{A_1}
  \ = \ -[\vartheta_1,\overline{\vartheta}_1] \,.
\end{equation}
This two form takes value in 
\[
  [\fg^{-1,1}_\varphi , \fg^{1,-1}_\varphi] \ \subset \ 
  \fg^{0,0}_\varphi \ = \ \fk_\bC^0 \,.
\]
The holomorphic sectional curvature of the period domain is negative and bounded away from zero in the horizontal directions \cite[Theorem 13.6.3]{MR3727160}.  Likewise, there exists $\epsilon > 0$ so that
\begin{equation}\label{E:trThetaD}
  \mathrm{tr}\,\Theta_D(v,\bar v)
  \ = \ -\mathrm{tr}\,\left(A_1(v)\,{}^t\overline{A_1(v)} \right)
  \ < \ -\epsilon\,h(v) 
  \,,\quad \forall \ v \in \left.\bfI\right|_D \,.
\end{equation}
Lemma \ref{L:cherna} implies there exist positive constants $\epsilon,\epsilon'$ so that
\begin{equation}\label{E:ThetaDw}
  -\epsilon\,\left.c_1(TD)\right|_{\bfI} \ \le \ 
  \left.c_1(\bfL)\right|_{\bfI} \ \le \ 
  -\epsilon'\,\left.c_1(TD)\right|_{\bfI} \,. 
\end{equation}

\subsubsection{Proof of Lemma \ref{L:VIII.10}} \label{S:prf10}

The key observation is that the restriction
\[
  \left.\tGr_{\cFe}^{-1}\tEnd(\cEe)\right|_B \ \simeq \ 
  \Phi^*(\bfI) \,.
\]
Set
\[
  \bfH \ = \ \tw^{\tdim\,B}\,\bfI \,,
\]
so that 
\begin{equation}\label{E:H2}
  \cH \ = \ \Phi^*(\bfH) \,.
\end{equation}
Since $\Theta_D$ and $\Theta_\bfH$ are homogeneous, Schur's lemma implies there exist positive constants $\epsilon_1, \epsilon_2$ so that 
\begin{equation}\label{E:c1c2}
  \epsilon_1\,\mathrm{tr}\,\left.\Theta_D\right|_\bfI
  \ \le \ 
  \mathrm{tr}\,\left.\Theta_\bfH\right|_\bfI
  \ \le \ 
  \epsilon_2\,\mathrm{tr}\,\left.\Theta_D\right|_\bfI \,.
\end{equation}
The lemma now follows from \eqref{E:ThetaDw} and \eqref{E:H2}.
\hfill\qed

\begin{remark}\label{R:chern}
The inequalities \eqref{E:c1c2} may also be deduced from Lemma \ref{L:cherna}.  Indeed, one may show that there exist positive constants $\epsilon,\epsilon'$ so that the the Chern forms satisfy 
\[
\begin{array}{rcccl}
  \epsilon \,\left.c_1(\bfL) \right|_\bfI 
  & \le & 
  -\left.c_1(TD) \right|_\bfI 
  & \le &
  \epsilon' \,\left.c_1(\bfL) \right|_\bfI  \,,\\
  \epsilon \,\left.c_1(\bfL) \right|_\bfI 
  & \le &
  -\left.c_1(\bfH) \right|_\bfI 
  & \le &
  \epsilon '\,\left.c_1(\bfL) \right|_\bfI  \,.
\end{array}
\]
\end{remark} 

\subsubsection{Proof of Lemma \ref{L:VIII.12a}} \label{S:prf12a}

The outline of the proof is sketched on page \pageref{sp:VIII.12a}; here we verify the details.

Fix a smooth local framing $\{\eta_j\}$ of 
\[ 
  \cH \ = \ \Phi^*(\bfH)
\]
over $B$ so that $\eta_1$ spans $(K_\olB+[Z])^* = \tw^{\tdim\,B}T_\olB(-\log Z)$, and is orthogonal to the $\{\eta_a\}_{a\ge2}$ with respect to the Hermitian form.  Let $\nabla \eta_j = \theta^i_j\ot\eta_i$ denote the local connection 1-forms, and 
\[
  \Theta_\cH \ = \ \Theta^i_j\,\eta_i\ot\eta^j
\]
the local curvature $(1,1)$--forms.  Then 
\[
  \left.\Theta_\cH\right|_{(K_\olB + [Z])^*} \ = \ \Theta^1_1 
  \tand
  \Theta_{(K_\olB + [Z])^*} \ = \ \td\theta^1_1 \,,
\]
are related by 
\[
  \Theta^1_1 \ = \ \td \theta^1_1 \ + \ \theta^1_j \wedge \theta^j_1 \,.
\]
The fact that $\eta_1$ is $h$--orthogonal to $\{\eta_a\}_{a\ge 2}$ implies that \[
  h_{11}\,\theta^1_a \ + \ h_{ab}\,\overline{\theta^b_1} \ = \ 0 \,,
\]
with $h_{11} = h(\eta_1,\eta_1)$ and $h_{ab} = h(\eta_a,\eta_b)$.  Since $(h_{ab})_{a,b\ge2}$ is positive definite, we see that 
\[
  \frac{\bi}{2\pi}\,\theta^1_j \wedge \theta^j_1 \ = \ 
  \frac{\bi h_{ab}}{2\pi h_{11}}\,
  \theta^a_1 \wedge \overline{\theta^b_1}
\]
is a non-negative $(1,1)$--form.  Setting
\begin{equation}\label{E:UU}
  (\Upsilon,\Upsilon) \ = \ 
  \frac{1}{h_{11}} h_{ab} \,\theta^a_1 \wedge \overline{\theta^b_1}
\end{equation}
yields \eqref{E:Up}. 

\begin{remark} \label{R:Up}
The endomorphism valued $(1,0)$--form 
\[
  \Upsilon \ = \ \theta^b_1 \,\eta_b \ot \eta^1
\]
is the second fundamental form \eqref{E:dfnUp} of $(K_\olB+[Z])^* \subset \cH$, \cite[\S4]{MR0282990}.
\end{remark}

It remains to show that $-\tfrac{\bi}{2\pi}h(\Theta\cdot\eta_1,\eta_1)$ is positive.  This is a consequence of the structure \eqref{E:ThetaU} and \eqref{E:ThetaI} of the curvature, and the IPR.  Given $b \in B$, the injectivity of the differential $\td\Phi$ and the IPR allow us to identify identify $T_bB$ with an abelian subalgebra $\fb \subset \fg^{-1,1}_\varphi$, $\tdim\,\fb = \tdim\,B$.  So for the purpose of this algebraic computation we may work point-wise and identify $\Theta_\cH$ with $\Theta = -[\vartheta_1,\overline\vartheta_1]$, with $\vartheta_1$ taking value $\fb$, and $\eta_1$ with an element of the line $\tw^{\tdim\,B}\fb \subset \tw^{\tdim\,B}\fg^{-1,1}_\varphi$.  The adjoint action $\tad:\fg_\bC \to \tEnd(\fg_\bC)$ induces an action of $\fg_\bC$ on $\tw^{\tdim\,B}\fg_\bC$.  The fact that $\fb$ is abelian implies $\vartheta_1\cdot \eta_1 = 0$, so that $[\vartheta_1,\overline\vartheta_1] \cdot \eta_1 = \vartheta_1 \overline\vartheta_1 \cdot \eta_1$ and 
\begin{eqnarray*}
  h(\Theta\cdot\eta_1,\eta_1) & = & 
  -h\left([\vartheta_1,\overline\vartheta_1]\cdot\eta_1 \,,\, \eta_1 \right)\\
  & = & 
  -h \left(\vartheta_1\,\overline\vartheta_1 \cdot\eta_1 \,,\, \eta_1 \right)
  \\ & = & 
  -h(\overline\vartheta_1 \cdot\eta_1\,,\,\overline\vartheta_1 \cdot\eta_1) \,.
\end{eqnarray*}
The final equality is the due to the fact that the $\overline\vartheta_1$ is the $h$--conjugate transpose of $\vartheta_1$, \cite[Corollary 12.6.3]{MR3727160}.

To see that $\overline\vartheta_1 \cdot \eta_1 \in \tw^{\tdim\,B}\fg_\bC$ is nonzero, recall that $\vartheta_1$ is nonzero (the Torelli hypothesis) and takes value in $\fb$.  We can complete $\fb$ to a basis $\{\vartheta_1,\xi_2,\ldots,\xi_d\}$ so that $\eta_1 = \vartheta_1\wedge \xi_2\wedge \cdots \wedge \xi_d$.  Keeping in mind that $[\overline\vartheta_1,\vartheta_1]$ is nonzero \cite[Corollary 12.6.3]{MR3727160}, we see that
\[
  \overline\vartheta_1 \cdot \eta_1 \ = \ 
  [\overline\vartheta_1,\vartheta_1] \wedge \xi_2 \wedge \cdots \wedge \xi_d
  \ + \ \sum_{j=2}^d \vartheta_1 \wedge \xi_2 \wedge\cdots
  \wedge [\overline\vartheta_1 , \xi_j ] \wedge \cdots \wedge \xi_d
\]
is nonzero.  This establishes the positivity of $h(\Theta\cdot\eta_1,\eta_1)$.

It now follows from the local Torelli assumption (\S\ref{S:loctor0} and Lemma \ref{L:ltc2}) that there exists $\epsilon > 0$ so that $\epsilon\,c_1(\Le) \le -\frac{\bi}{2\pi} h(\Theta\cdot\eta_1,\eta_1)$.  A priori this $\epsilon$ depends on our choice of $b \in B$.  However, homogeneity under the action of $G_\bR$ and the fact that the grassmannian $\tGr(\tdim\,B,\fg^{-1,1}_\varphi) \ni \fb$ is compact imply that we can find $\epsilon>0$ that works for all $b\in B$.  \hfill\qed

\subsubsection{Proof of Lemma \ref{L:VIII.12b}} \label{S:prf12b}

As indicated in the sketch of the proof (page \pageref{sp:VIII.12b}) it suffices to show that $\left.\Upsilon\right|_\fibrezero$ may be identified with the differential of the Gauss map $\cG(\left.\Phione\right|_\fibrezero)$.  To see this recall the set-up of \S\ref{S:prf12a}.  The local section $\eta_1$ of $(K_\olB+[Z])^*$ defines (globally) a map
\[
  [\eta_1] : \olB \ \to \ \bP(\tw^{\tdim\,B}T_\olB(-\log Z)) \,.
\]
The derivative of $[\eta_1]$ is 
\begin{eqnarray*}
  \td [\eta_1] &  = & \nabla \eta_1 \quad\hbox{mod} \quad (K_\olB+[Z])^*\\
  & = & \sum_{a\ge2} \theta^a_1\,\eta_a \quad\hbox{mod} \quad (K_\olB+[Z])^*
\end{eqnarray*}
and may be identified with $\Upsilon$ (defined in Remark \ref{R:Up}).  We may then identify the differential $\td \cG(\left.\Phione\right|_\fibrezero)$ with the restriction of $(\theta^a_1)_{a\ge 2}$ to $\fibrezero$.  It then follows from \eqref{E:UU} that $\left.\tau\right|_\fibrezero$ is positive if and only if $\td \cG(\left.\Phione\right|_\fibrezero)$ is injective. \hfill \qed

\begin{remark}
It follows from \cite[Remark 5.1]{GGRhatPT} and Lemma \ref{L:ltc2} that $[\eta_1]$ may be thought of as a Gauss map for $\PhiT$.  
\end{remark}

\def\cprime{$'$} \def\Dbar{\leavevmode\lower.6ex\hbox to 0pt{\hskip-.23ex
  \accent"16\hss}D}

\end{document}